\documentclass[10pt,a4paper]{amsart}
\usepackage{amstext, amssymb, amsthm, amsfonts, latexsym,bbm,tikz,sfmath,color,varioref}
\usetikzlibrary{matrix,positioning}

\usepackage[version=4]{mhchem}

\usepackage{cancel}

\usepackage{enumerate}
\usepackage[all]{xy}
\usepackage{graphicx}
\usepackage{pifont,mathtools}
\usepackage{hyperref}

\usepackage{float}

\usepackage{setspace}
\usepackage{dsfont}
\usepackage{geometry}
\usepackage{enumitem}
\usepackage{todonotes}
\usepackage{bbm}
\usepackage[utf8]{inputenc}
\usepackage[english]{babel}
\usepackage{textcomp}

\usepackage{comment}
\usepackage{tikz-cd}
\usepackage{hyperref}

\let\emptyset\varnothing

\makeatletter
\newcommand\level[1]{%
  \ifcase#1\relax\expandafter\chapter\or
    \expandafter\section\or
    \expandafter\subsection\or
    \expandafter\subsubsection\else
    \def\next{\@level{#1}}\expandafter\next
  \fi}
\newcommand{\@level}[1]{%
  \@startsection{level#1}
    {#1}
    {\z@}%
    {-3.25ex\@plus -1ex \@minus -.2ex}%
    {1.5ex \@plus .2ex}%
    {\normalfont\normalsize\bfseries}}

\newcounter{level4}[subsubsection]
\@namedef{thelevel4}{\thesubsubsection.\arabic{level4}}
\@namedef{level4mark}#1{}
\count@=4
\loop\ifnum\count@<100
  \begingroup\edef\x{\endgroup
    \noexpand\newcounter{level\number\numexpr\count@+1\relax}[level\number\count@]
    \noexpand\@namedef{thelevel\number\numexpr\count@+1\relax}{%
      \noexpand\@nameuse{thelevel\number\count@}.\noexpand\arabic{level\number\numexpr\count@+1\relax}}
    \noexpand\@namedef{level\number\numexpr\count@+1\relax mark}####1{}}
  \x
  \advance\count@\@ne
\repeat
\makeatother
\newtheorem{theorem}{Theorem}[section]
\newtheorem{lemma}[theorem]{Lemma}

\newtheorem{proposition}[theorem]{Proposition}

\theoremstyle{definition}
\newtheorem{definition}[theorem]{Definition}

\newtheorem{example}[theorem]{Example}

\theoremstyle{remark}
\newtheorem{remark}[theorem]{Remark}

\theoremstyle{consequence}

\def\la{\lambda}

\def\1{\mathbf{1}}
\def\la{\lambda}

\def\p{^{\prime}}

\def\k{\kappa}

\def\cN{\mathcal{N}}
\def\cR{\mathcal{R}}
\def\cS{\mathcal{S}}
\def\cC{\mathcal{C}}

\def\E{\mathbb{E}}
\def\R{\mathbb{R}}

\def\Z{\mathbb{Z}}

\numberwithin{equation}{section}

\begin{document}

\title[Identifiability of SDEs for RNs]{\vspace*{-1.4cm}Identifiability of SDEs for reaction networks}

\author{
Louis Faul\textsuperscript{1}, 
Linard Hoessly\textsuperscript{2}, 
Panqiu Xia\textsuperscript{3}
}

\thanks{\textsuperscript{1} Department of Mathematics, University of Fribourg, Fribourg, Switzerland. Email: \href{mailto:louis.faul@unifr.ch}{louis.faul@unifr.ch}}
\thanks{\textsuperscript{2} Clinic for Transplantation Immunology and Nephrology, Basel University Hospital, Basel, Switzerland. Email: \href{mailto:linarddavid.hoessly@usb.ch}{linarddavid.hoessly@usb.ch}}
\thanks{\textsuperscript{3} School of Mathematics, Cardiff University, Cardiff, United Kingdom. Email: \href{mailto:xiap@cardiff.ac.uk}{xiap@cardiff.ac.uk}}

\date{}
\begin{abstract}
Biochemical reaction networks are widely applied across scientific disciplines to model complex dynamic systems. We investigate the diffusion approximation of reaction networks with mass-action kinetics, focusing on the identifiability of the
stochastic differential equations associated to the reaction network.
  We derive conditions under which the law of the diffusion approximation is identifiable and provide theorems for verifying identifiability in practice. Notably, our results show that some reaction networks have non-identifiable reaction rates, even when the law of the corresponding stochastic process is completely known. Moreover, we show that reaction networks with distinct graphical structures can generate the same diffusion law under specific choices of reaction rates. Finally, we compare our framework with identifiability results in the deterministic ODE setting and the discrete continuous-time Markov chain models for reaction networks.
 \end{abstract}


\clearpage\maketitle
\thispagestyle{empty}
\vspace*{-0.6cm}


\section{Introduction}\label{summary_1}

Reaction networks (RNs) are a fascinating theme, which offers both a framework and structures to describe biochemical dynamics at different scales.
They are applied across disciplines like ecology \cite{ecol_May}, sociology \cite{sociology}, cancer \cite{cancer1,cancer2}, or neural networks \cite{Goutsias_ov}. Moreover, mathematical research on RNs drives both different fields of mathematics and their interplay, including dynamical systems, probability, statistics, combinatorics, or applied algebraic geometry \cite{anderson2,craciun2,overv_Gorban,HOESSLY_WIUF_XIA_2023,approx_kurtz,Hoessly_esc}.

RNs are mostly given via their reaction graph, where arrows correspond to reactions. As an example, consider the following RN, which serves as a simplified model for gene expression \cite{Thattai8614}
 \begin{equation} \label{eq_gene}\emptyset \rightleftharpoons S_1,\quad S_2\to \emptyset,\quad S_1\to S_1+S_2.\end{equation}
 While such a simple example already displays complex behavior with the exact stationary distribution currently unknown \cite{Thattai8614}, realistic RNs of interest are usually fairly big, with many reactions and uncertainty in parameter values. Correspondingly, methods to understand properties of their dynamics based on the RNs are a main focus.

 These dynamics can be studied in three different settings depending on the number of molecules. When the number of molecules is low, the state of the system is modeled stochastically using a Continuous Time Markov Chain (CTMC) \cite{approx_kurtz,mcquarrie_1967}.
{In the reaction-network and chemical-physics literature, this same mathematical object is often referred to as a \emph{jump process}, whose construction from a RN is governed by the rules of \emph{stochastic chemical kinetics} and whose sample paths are typically simulated using variants of \emph{Gillespie’s algorithm} \cite{gillespiea,gillespieb}.} If there are a lot of molecules, species abundances are rescaled into concentrations, and studied through a system of ordinary differential equations \cite{craciun2,Feliuinj,Michaelis_M,Wegscheider}. 
 On the intermediate scale, a diffusion approximation is used \cite{Kurtz1976,KURTZ1978,approx_kurtz} via stochastic differential equations (SDEs). This formalism still allows the stochasticity of the system to be taken into account and speeds up the simulations. Unfortunately the two commonly used diffusion approximations namely the linear noise approximation (LNA), and the Langevin approximation (LA), don't satisfy the natural condition that the abundances stay nonnegative for all time, and a reflected diffusion, the so-called Constrained Langevin approximation (CLA), has therefore been introduced recently \cite{leite2019}.

A crucial aspect of such quantitative methods is the ability to match the models to the available data. In the case of the ODE model, this can simply be done e.g. via least squares. However, for the stochastic models, statistical inference is required, which tends to be more challenging \cite{mazza,wilkinson2018stochastic}. For both, an often overlooked issue in practice is identifiability, where a statistical model is called identifiable if its parametrisation map $\theta\mapsto p_\theta$ is injective. For RNs, there are two settings of interest: reaction-identifiability i.e. identifiability of rate constants given the RN structure, and confoundability e.g. between two different RNs. In practice, both can be necessary for proper parameter inference. Identifiability has been analytically studied in the deterministic setting \cite{ident_CRNs}, where depending on the RN both rate constants as well as the RN structure are potentially not identifiable. In CTMC models, both RN structure and rate constants can be identified as long as the observed state space available is big enough \cite{enciso2020identifiability}. Concerning diffusion approximation, several works have investigated related questions, e.g.
identifiability by moment equations \cite{ident_review} for the Langevin equation, identifiability through Fisher information matrix \cite{ident_LNA} for the LNA, or investigations that consider stationary distributions of the linear noise approximation (LNA) \cite{grunberg2023identifiability}. However, identifiability has not been analytically studied for RNs w.r.t. their LAs. 

Identifiability can generally be divided into \emph{structural identifiability}, which is a property of the model and observation map under the assumption of ideal and complete information; and \emph{practical identifiability}, which accounts for inference from real-world data subject to limitations, incompleteness, and measurement or observation errors. In the following, we focus on the analytical characterisation of \emph{structural identifiability} for RNs modelled by SDEs. Since SDE models are typically obtained through a finer scaling limit than that used for deriving the corresponding ODE \cite{approx_kurtz}, it is reasonable to conjecture that the ODE is also coarser concerning identifiability. We confirm this conjecture and characterise conditions for identifiability of diffusion approximations through identifiability of the SDEs. Throughout the manuscript we assume mass-action kinetics, i.e., kinetics where the reaction rates are proportional to the amounts of the reactants, though other kinetics are possible as well (e.g., Hill kinetics type I/II, or Michaelis-Menten, cf. \cite{anderson2,reg_cellular,horn3,grima_kin}).

We present our results for SDEs without reflecting boundary conditions focusing on the standard Langevin approximation (LA) \cite{Kurtz1976,KURTZ1978,approx_kurtz}. As our analysis is framed in terms of the generators of the SDEs (see Section \ref{identif_generator}), the results are likely extendable to SDEs with reflecting boundary conditions, i.e., the CLA \cite{anderson2019constrained,leite2019}. LA and CLA share the same differential operator as their generator, differing only in the set of test functions: for CLA, these functions must satisfy additional boundary conditions. Although this extension falls outside the present scope, it offers a natural avenue for future exploration.

To motivate our investigation, we give the following simple example illustrating that the SDE associated with a reaction network may not be identifiable. Specifically, in this example, two reaction networks with the same structure but different reaction rates can give rise to identical SDEs, making them indistinguishable based on observed process realisations. 

\begin{example}\label{1d_non_ident_generator_intro}
Let $S$ be a chemical species and consider the following RN with two different sets of rate constants:
$$2S\xleftarrow{1}\emptyset\xrightarrow{4}S\xrightarrow{1} \emptyset\xrightarrow{2}3S,$$
$$2S\xleftarrow{4}\emptyset\xrightarrow{1}S \xrightarrow{1} \emptyset\xrightarrow{1}3S .$$
Let $s$ denote the abundance of species $S$. In the deterministic setting both RNs give the same ODE:
$$\frac{ds}{dt}=12-s.$$
Hence the rate constants aren't reaction-identifiable w.r.t. the ODE. Concerning their SDEs, one can calculate their generator which is determined by the drift vector and diffusion matrix (see \eqref{drift_vector} and \eqref{diffusion_matrix}), which are for both RNs given by:
$$A(s)=12-s,\quad B(s)=s+26 .$$
The two RNs give the same generator and thus the reaction rates aren't reaction-identifiable w.r.t. their generators. 
As the corresponding SDE is regular enough and since the generator determines the law, we conclude that the rate constants are not reaction-identifiable w.r.t. the law of the SDE. 
\end{example}

{\subsection*{Outline}
In Section~2 we recall basic notions on reaction networks with mass-action kinetics, and present the SDE corresponding to the Langevin approximation. In Section~3 we define reaction-identifiability and (un-)confoundability with respect to SDE dynamics and state our main characterisations, together with an extension to linear conjugacy. Section~4 is devoted to examples: we discuss special classes of reaction-identifiable networks and provide counterexamples. In Section~5 we reformulate identifiability in terms of generators and establish the link between identifiability of SDEs and of their generators. Section~6 contains concluding remarks and perspectives, while all proofs of the main results are collected in Section~7.
}

\section{Preliminaries}\label{RN_intro}

\subsection{Notations}
Let $\mathbb{Z}^m_{\geq 0}$ represent the m-dimensional lattice of nonnegative integers, $\mathbb{R}^m_{\geq 0}$ the subset of $\mathbb{R}^m$ whose coordinates are all nonnegative, and $\mathbb{R}^m_{> 0}$ the subset of $\mathbb{R}^m$ whose coordinates are positive.
The set of $n \times m$ matrices with real entries is written $\mathbb{R}^{n} \otimes \R^m$, where the transpose of a matrix $A$ or a vector $x$ is denoted $A^{\intercal}$ and $x^{\intercal}$. For two vectors $x, y \in \mathbb{R}^m$ (viewed as $m \times 1$ column matrices), we denote by $x \cdot y^{\intercal}$ the resulting $m \times m$ matrix product.
For a set $U \subseteq \mathbb{R}^m$, let 
$C_c ^2(U)$ be the set of twice continuously differentiable functions that have
compact support on $U$. Similarly, $C_c ^\infty(U)$ represents the set of infinitely continuously differentiable functions that have
compact support on $U$.
Let $X$ and $Y$ be two random variables; we write $X \overset{d}{=} Y$ if they have the same distribution.

\subsection{Reaction networks}

In this subsection, we provide a brief introduction to the concepts of RNs and derive both deterministic and stochastic formulations of their dynamics. For a comprehensive treatment of this topic, we refer the reader to \cite{anderson2015stochastic,feinberg2019foundations}.

\begin{definition}\cite{feinberg2019foundations}
    A \textbf{reaction network }(RN) is a triple $(\cS,\cC,\cR) \eqqcolon \cN$, where
 $\cS$ is the set of \textbf{species} $\cS=\{S_1,\dots,S_n\}$, $\cC = \{C_1, \dots, C_m\}$ is the set of \textbf{complexes}, and $\cR$ is the set of \textbf{reactions} $\cR=\{r_1,\dots,r_d\}$. The set of \textbf{species} is considered as an orthogonal basis in the Euclidean space $\R^n$, meaning that each species is a basis vector.  
 A \textbf{complex} is a nonnegative integer linear combination of species, represented as a vector in $\Z_{\geq 0}^n$. That is, for each $y \in \cC$, we express
\[
y = (y_1,\dots, y_n) = \sum_{i=1}^n y_i S_i.
\]
A \textbf{reaction} is an ordered pair $(y, y\p)\in \cR$ with $y,y\p\in\cC$ and is typically written as $y\to y\p $. In this reaction the \textbf{source complex} $y$ is consumed and the \textbf{product complex} $y\p$ is produced.
 \end{definition}

 The \textbf{molecularity} of a reaction refers to the total number of reactant molecules involved in an elementary reaction step.  Specifically, for a reaction $y\to y\p \in\cR$, the molecularity is given by $|y| = \sum\limits_{i=1}^n y_i$. Correspondingly we call such reactions unimolecular, bimolecular, three-molecular or $p$-molecular reactions provided $|y| = 1, 2, 3$, or $p$.
The \textbf{reaction vector} for a reaction $r = y \to y\p$ is defined as $l_r = l_{y \to y'} = y' - y \in \mathbb{Z}^n$. 
The \textbf{stoichiometric matrix} $\Gamma$ is an $n \times d$ matrix, where each column represents the reaction vector of a corresponding reaction.

Assuming the law of mass-action, each reaction $r = y\to y\p$ is assigned a positive  \textbf{rate constant} 
$\k_r = \k_{y\to y\p}$. The \textbf{rate constant vector} is defined as 
\[
  \k = (\k_1,\dots, \k_d) = (\k_{r_1},\dots, \k_{r_d}) \in\R_{>0}^d.
  \]

Reaction network dynamics can be modeled in various ways. For large molecule counts,  \textbf{ordinary differential equations} (ODEs) are commonly used.  Let $x(t)$ denote the species concentration vector, a time-dependent function in $\R^n$. Then, it satisfies
\begin{equation}\label{ode}
\dot{x} = F(x) =  \sum\limits_{r=1}^d l_r \lambda_r(x) , 
\end{equation}
where $\lambda_r(x)$ is the \textbf{kinetics}---the reaction rate of $y_r \to y_r\p$ evaluated at $x \in \mathbb{R}^n_{\geq 0}$. Under the law of mass-action, the kinetics (\textbf{mass-action kinetics}) with rate constant $\k_r$ is written as
\begin{align}\label{def_ma-kinetics}
\lambda_r(x) = \kappa_r x^{y_r} =\kappa_r \prod\limits_{i=1}^n x_i^{y_{i,r}}. 
\end{align}

In contrast, in a system of  small molecule counts, random fluctuations are significant, and thus \textbf{continuous time Markov chains} (CTMCs) are used to model the dynamics of a reaction network. The species counts, denoted by $X_t$, is a $\Z_{\geq 0}^n$-valued CTMC, satisfying the equation
\[
  X(t) = X(0) + \sum_{r = 1}^d Y_{r} \Big(\int_0^t \widetilde{\lambda}_{r} (X(s))\Big) l_r,
\]
where $\{Y_1,\dots, Y_d\}$ is a set of i.i.d. Poisson processes with unit rates, and under the law of mass-action, the kinetics in the stochastic modeling with rate constant $\k_r$ is given by
$$\widetilde{\lambda}_r(x)  = \kappa_r \prod\limits_{i=1}^n \frac{x_i!}{(x_i - y_{i,r})!}1_{\{x_i\geq y_{i,r}\}}.$$
Throughout this article, we denote by $(\cN, \k)$ for an RN $\cN$ equipped with mass-action kinetics specified by the rate constant vector $\k$.

We conclude this subsection by introducing the definition of subnetworks.
\begin{definition} \cite{fontanil2021common}
    Let $\mathcal{N}=(\mathcal{S},\mathcal{C},\mathcal{R})$ be an RN. Then, any reaction network $\mathcal{N}'=(\mathcal{S}',\mathcal{C}',\mathcal{R}')$ is a subnetwork of $\mathcal{N}$ if $\mathcal{R}' \subseteq \mathcal{R}$ ({which implies both $\mathcal{S}' \subseteq \mathcal{S}$ and $\cC' \subseteq \cC$}).
\end{definition}

\subsection{Langevin approximation}

In practice, the CTMCs are typically studied through Monte Carlo simulations. Since the discrete stochastic simulations are generally computationally expensive, diffusion approximations, such as \textbf{Linear Noise approximation} (LNA) \cite{vankampen1} and \textbf{Langevin approximation} (LA) \cite{Kurtz1976,KURTZ1978}, are commonly used in numerical computation. Recently, a  constrained Langevin approximation \cite{anderson2019constrained,leite2019} was developed to address issues arising when the approximating process in LNA or LA reaches the boundary of the positive orthant. In this article, we focus on the identifiability of LA for stochastically modelled  RNs.

As a frequently used diffusion approximation for stochastically modelled RNs, the LA is derived from the \textbf{chemical Langevin Equation} (CLE) \cite{Gillespie_di}, an It\^{o}-type \textbf{stochastic differential equation} (SDE). The LA provides an accurate approximation for stochastic RNs up to the first hitting time of the boundary of the positive orthant and is preferred for capturing nonlinear effects in random fluctuations. Specifically, for an RN $(\cN, \k)$, the LA satisfies the following $n$-dimensional CLE:
\begin{align}\label{Ito_sde} 
  d X_t = A(X_t) dt + \sigma(X_t) dW_t,
\end{align}
for all $t > 0$ with an initial value $X_0 \in \R_{> 0}^n$, where $A \colon \R^n \to \R^n$ is the \textbf{drift vector}, $\sigma \colon \R^n \to \R^{n} \otimes \R^{n}$ is the \textbf{diffusion coefficient}, and $W$ is a $n$-dimensional Brownian motion. In the context of LA for stochastic RNs, $A$ is given by
\begin{equation}\label{drift_vector}
      A_i(x) = \sum_{r=1}^d \Gamma_{i,r}\la_r(x) = \sum_{r=1}^d (y'_{i,r} - y_{i,r}) \kappa_r x^{y_r}, \quad i = 1,\dots, n;
\end{equation}
and $\sigma$ is given as the unique positive semi-definite square root (existence and uniqueness of square roots for any symmetric positive semi-definite matrix is well-established; cf. \cite{bhatia97}.) of the \textbf{diffusion matrix} $B$, with entries
\begin{equation}\label{diffusion_matrix}
      B_{i,j}(x)=
(\sigma(x) \sigma(x)^{\intercal})_{i,j} = \sum_{r=1}^d \Gamma_{ir}\Gamma_{jr} \la_r(x) = \sum_{r = 1}^d (y'_{i,r} - y_{i,r}) (y'_{j,r} - y_{j,r}) \kappa_r x^{y_r}, \quad i, j = 1, \dots, n.
\end{equation}
Additionally, we recall that in \eqref{drift_vector} and \eqref{diffusion_matrix}, $\Gamma$ denotes the stoichiometric matrix, and $\lambda$ represents the mass-action kinetics defined in \eqref{def_ma-kinetics}. In particular, when the initial state is important, we use the notation $X_t(x)$ to denote the solution to \eqref{Ito_sde} with $X_0 = x$.

\section{Identifiability of RNs given dynamics}\label{ident_cons}

Although solutions to \eqref{Ito_sde} may lack uniqueness and/or may explode in finite time, they are well-defined and unique up to a stopping time when restricted to a bounded open subset of the positive orthant (see Subsection \ref{SDEs_theory}).  This property enables us to explore the identifiability of RNs given their dynamics in this section. We begin with some definitions.
 \begin{definition}\label{def_fidis}
     Two stochastic processes $(X_t)_{t\geq0},({X}'_t)_{t\geq0}$ have the same distribution if they have the same finite dimensional distributions (fidis), which is denoted $(X_t)_{t\geq0}\overset{d}{=}({X}'_t)_{t\geq0}$.
 \end{definition}

\begin{definition}
Let $\cN = (\cS, \cC, \cR)$ and $\cN' = (\cS, \cC', \cR')$ be two RNs with the same set of species, where $\cR \neq \cR'$.
\begin{itemize}
  \item  $\cN$ is \textbf{reaction-identifiable} w.r.t. its SDE, if for every non-empty open bounded set $U$, for all initial states $x\in U$ and distinct rate constant vectors $\k, \k'\in\R_{>0}^d$, the corresponding stochastic processes, when stopped upon reaching the boundary of $U$, have different distributions starting from $x$.

  \item   $\cN$ and $\cN'$ are \textbf{unconfoundable} w.r.t. their SDEs, if for every non-empty open bounded set $U$, for all initial states $x\in U$ and rate constant vectors $\k\in \R_{>0}^d$ and $\k'\in \R_{>0}^{d'}$ , the corresponding stochastic processes, when stopped upon reaching the boundary of $U$, have different distributions starting from $x$.
\end{itemize}
    
 \end{definition}

The next theorem establishes sufficient and necessary conditions for the identifiability of RNs w.r.t. their SDEs.
 The proof is provided in Subsection \ref{prf_weak_react_id}.

\begin{theorem}\label{weak_react_id}
Let $\cN = (\cS, \cC, \cR)$ and $\cN' = (\cS, \cC', \cR')$ be two RNs with the same set of species, where $\cR \neq \cR'$.

\begin{enumerate}
\item $\cN$ is reaction-identifiable w.r.t. its SDE if and only if for every source complex $y\in \cC$, the set of vectors 
\begin{equation}\label{generalized_reaction_vectors}
    \left\{\left(y'-y,(y' - y)\cdot (y' - y)^{\intercal}\right) \big| \, y\to y'\in\cR\right\}
\end{equation}
is linearly independent.

\item $\cN$ and $\cN'$ are unconfoundable w.r.t. their SDEs if and only if they have different source complexes or share the same source complexes and for every source complex $y \in \cC$, ${\rm Cone}_\cR(y) \cap {\rm Cone}_{\cR'}(y) = \emptyset$,
where, 
\begin{align}\label{def_cone}
{\rm Cone}_\cR(y)\coloneqq \left\{\sum_{y\to y'\in \cR}\alpha_{y\to y'}\left((y' - y),(y' - y)\cdot(y' - y)^{\intercal}\right)\bigg|\, \alpha_{y\to y'}>0\right\} \subseteq \R^n \times \R^n \otimes \R^n.
\end{align}
\end{enumerate}

\end{theorem}

\begin{example}\cite[Fig. 3]{ident_CRNs}\label{e.ident_diff}
    Consider the following two RNs, where the first RN has reaction set $\mathcal{R}$ and the second has reaction set $\mathcal{R}'$:
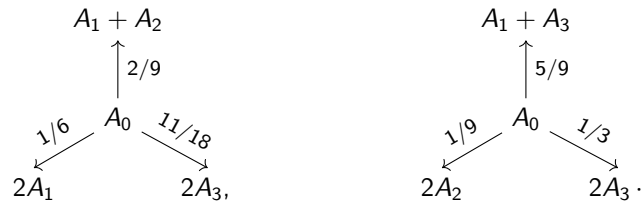
    \begin{center}
\begin{tikzpicture}
  \matrix (m) [matrix of math nodes, row sep=2.5ex]{
    & \ce{$A_1 + A_2$}  &\\
    &  & & & \\
    & \ce{$A_0$} &  \\
    \ce{$2 A_1$} & & \ce{$2 A_3$,}\\
  };
  \draw[->] (m-3-2) -- (m-1-2.south)
    node[right,pos=0.5,sloped,font=\footnotesize,rotate=270] {\ce{$2/9$}};

  \draw[->] (m-3-2) -- (m-4-1.north)
    node[above,pos=0.5,sloped,font=\footnotesize] {\ce{$1/6$}};
  \draw[->] (m-3-2) -- (m-4-3.north)
    node[above,pos=0.5,sloped,font=\footnotesize] {\ce{$11/18$}};

\end{tikzpicture}
\hspace{5em}
\begin{tikzpicture}
  \matrix (m) [matrix of math nodes, row sep= 2.5ex]{
    & \ce{$A_1 + A_3$}  &\\
    &  & & & \\
    & \ce{$A_0$} &  \\
    \ce{$2 A_2$} & & \ce{$2 A_3$.}\\
  };
  \draw[->] (m-3-2) -- (m-1-2.south)
    node[right,pos=0.5,sloped,font=\footnotesize,rotate=270] {\ce{$5/9$}};

  \draw[->] (m-3-2) -- (m-4-1.north)
    node[above,pos=0.5,sloped,font=\footnotesize] {\ce{$1/9$}};
  \draw[->] (m-3-2) -- (m-4-3.north)
    node[above,pos=0.5,sloped,font=\footnotesize] {\ce{$1/3$}};

\end{tikzpicture}
\end{center}
As shown in \cite{craciun}, these RNs are confoundable w.r.t. their ODEs. 
However, one finds that
    \begin{equation*}
        {\rm Cone}_\mathcal{R}(A_0) = \left\{\left(\begin{pmatrix}
            -\alpha_1-\alpha_2 -\alpha_3 \\
            \alpha_1+2\alpha_2\\
            \alpha_1\\
            2\alpha_3
        \end{pmatrix},\begin{pmatrix}
\alpha_1+\alpha_2+\alpha_3 & -\alpha_1-2 \alpha_2 & -\alpha_1  & -\alpha_3 \\
-\alpha_1-2\alpha_2   & \alpha_1+4\alpha_2 & \alpha_1  & 0\\
 -\alpha_1 & \alpha_1  &  \alpha_1& 0 \\
 -2\alpha_3 & 0 & 0 & 4 \alpha_3
\end{pmatrix}\right) \bigg|\, \alpha_1, \alpha_2, \alpha_3 >0 \right\} ,
    \end{equation*}
    while
    \begin{equation*}
        {\rm Cone}_\mathcal{R'}(A_0) = \left\{\left(\begin{pmatrix}
            -\alpha_1-\alpha_2 -\alpha_3 \\
            \alpha_1\\
            2\alpha_2\\
            \alpha_1+2\alpha_3
        \end{pmatrix},\begin{pmatrix}
\alpha_1+\alpha_2+\alpha_3 & -\alpha_1 & -2\alpha_2 & -\alpha_1-2\alpha_2\\
-\alpha_1 & \alpha_1& 0 & \alpha_1 \\
 -2\alpha_2 & 0  &  4\alpha_2& 0 \\
 -\alpha_1-2\alpha_2 & \alpha_1& 0 & \alpha_2+4\alpha_3
\end{pmatrix}\right) \bigg| \, \alpha_1,\alpha_2, \alpha_3 >0  \right\} .
    \end{equation*}
    It is easy to verify that ${\rm Cone}_\mathcal{R}(A_0) \cap {\rm Cone}_\mathcal{R'}(A_0) = \emptyset$, since some coefficients are zero in one of the above matrices but positive in the other.
    As a result of Theorem \ref{weak_react_id}, these two RNs are unconfoundable w.r.t. their SDEs.
\end{example}
\begin{remark}
\label{rem:nontrivial}
For a reaction network with mass-action kinetics, the rate constants associated
with a source complex $y\in \cC$ enter the CLE in both the drift and the diffusion terms,
\[
A(x)
= \sum_{y\in \mathcal{C}}\sum_{y\to y'\in R} (y'-y)\,\kappa_{y\to y'}\, x^y,
\qquad
B(x)
= \sum_{y\in \mathcal{C}}\sum_{y\to y'\in R} (y'-y)\cdot(y'-y)^{\intercal}\,\kappa_{y\to y'}\, x^y.
\]

Hence, for each source complex $y\in C$, the rate constants appear only through the
finite family
\[
\bigl( y'-y,\,(y'-y)\cdot(y'-y)^{\intercal} \bigr),
\qquad y\to y'\in \mathcal{R}.
\]
Theorem~\ref{weak_react_id} shows that reaction-identifiability of the SDE is
\emph{equivalent} to the linear independence of this family.

On the other hand, in the deterministic ODE setting, identifiability depends solely on the reaction
vectors $y'-y$~\cite{craciun2}. In contrast, in the SDE
setting the diffusion matrix contributes additional constraints through the terms
$(y'-y)\cdot(y'-y)^{\intercal}$, so that identifiability relies on the combined information from both drift and diffusion.

For a general parametric SDE that does not arise from a reaction network, one may still ask whether the map $\theta \mapsto (A_\theta,B_\theta)$ is injective. However, one cannot typically expect the existence of a canonical finite family of vectors whose linear independence characterises identifiability \cite{SDE_inference}. The mass-action parametrisation of reaction networks is therefore crucial for the simple linear-algebraic criterion established in Theorem~\ref{weak_react_id}.
\end{remark}

\subsection{Linear conjugacy}
In the context of RNs, the concept of linear conjugacy was introduced in \cite{johnston2011linear} for deterministically modeled RNs, borrowing ideas from dynamical systems theory \cite{perko2001differential}. This notion was further developed in subsequent works \cite{johnston2012linear, johnston2013computing, nazareno2019linear}, among others. Earlier work on conjugacy in detailed balanced RNs can be found in \cite{krambeck1970mathematical}. Intuitively, two RNs are said to be linearly conjugated if there exists a linear mapping that transforms the trajectories of one system into those of the other. This property is of particular interest because qualitative features of mass-action systems, such as local stability, multistability, or persistence, are preserved under linear conjugacy.
Our goal is to adapt the concept of linear conjugacy and extend the results of \cite{johnston2011linear} to the stochastic setting.

\begin{definition}
 Let $(\cN, \k)$ and $(\cN', \k')$ be two RNs with mass-action kinetics, that share a common set of species and $\cR \neq \cR'$.
    \begin{itemize}
      \item $(\cN, \k)$ and $(\cN', \k')$ are \textbf{$C^k$-conjugated} w.r.t. their SDEs, if there exists a $C^k$-diffeomorphism $h \colon \R_{>0}^n \to \R_{> 0}^n$, such that for any bounded open set $U \subseteq \R_{> 0}^n$, and every $x \in U$,
      \begin{align}\label{linear_map}
      h(X_{t \wedge \tau}(x)) \overset{d}{=} X'_{t \wedge \tau'} (h(x)), 
      \end{align}
      where $X_{t \wedge \tau}  (x)$ and $X_{t\wedge \tau'}' (h(x))$ denotes the solution to the LAs for $(\cN, \k)$ and $(\cN',\k')$ starting at $x$ and $h(x)$ respectively, up to stopping times defined as
      \begin{align}\label{def_tau-tau'}
        \tau \coloneqq \inf \{t > 0, X_t (x) \notin U\}\quad \text{and} \quad \tau' \coloneqq \inf \big\{t > 0, X_t'(h(x)) \notin h(U)\big\}.
        \end{align}

        \item $(\cN, \k)$ and $(\cN', \k')$  are \textbf{linearly-conjugated} w.r.t. their SDEs, if they are $C^{\infty}$-conjugated w.r.t. their SDEs and the diffeomorphism $h$ is linear.
    \end{itemize}
     
\end{definition}

Analogous to \cite{johnston2011linear}, the following theorem provides a necessary and sufficient condition for the linear conjugacy of two distinct RNs that share the same set of species. The proof is deferred to Subsection~\ref{prf_linear-conj}.
\begin{theorem}\label{thm_linear-conj}
  Let $\cN$ and $\cN'$ be two RNs, that share a common set of species and $\cR \neq \cR'$. Then, there exist rate constant vectors $\k$ and $\k'$ such that $(\cN, \k)$ and $(\cN', \k')$ are linear{ly} conjugate{d} w.r.t. their SDEs if and only if there exists a matrix $G$ of the form $G = DP$, where $D$ is a
positive diagonal matrix and $P$ is a permutation matrix, such that $\mathcal{N}$ and $\mathcal{N}'$ have the same complexes up to permutation by P,
and such that for every source complex $y$ of $\mathcal{N}$ it holds that:
  \[
    {\rm Cone}_{\cR} (y) \cap  {\rm Cone}^G_{\cR'} (y) \neq \emptyset,
  \] 
  where ${\rm Cone}_{\cR}$ is given by \eqref{def_cone}, and
  \[
    {\rm Cone}^G_{\cR'} (y) \coloneqq \left\{\sum_{y\to y'\in \cR'}\alpha_{y\to y'}  \left(G(y' - y),G (y' - y) \cdot (y' - y)^{\intercal} G^{\intercal}\right)\bigg|\, \alpha_{y\to y'}>0\right\} .
  \]
\end{theorem}

Suppose the matrix $G$ in Theorem \ref{thm_linear-conj} consists of only positive scaling. Then, $G$ is a positive diagonal matrix, with diagonal entries $\{c_1, \dots, c_n\} \subseteq \R_{> 0}$. The following proposition describes the relationship between the rate constant vectors of linearly conjugate RNs w.r.t. their corresponding SDEs. The proof is presented in Subsection \ref{prf_linear-conj-rate}.

\begin{proposition}\label{prop_linear-conj-rate}
    Let $\cN = (\cS, \cC, \cR)$ and $\cN' = (\cS, \cC', \cR')$ be two RNs with the same set of species and source complexes. Suppose that for the rate constant vector $\k \in \R_{> 0}^d$, there exists a positive vector $\beta \in \R_{> 0}^{d'}$ and a positive diagonal matrix $G$ with diagonal entries $\{c_1, \dots, c_n\} \subseteq \R_{> 0}$, such that for every source complex $y$:
    \begin{align}\label{eq_linear-conj-1}
    \sum_{y' \colon y\to y'\in \cR}\k_{y\to y'}(y' - y) =\sum_{y' \colon y\to y'\in \cR'}\beta_{y\to y'} G (y' - y) 
    \end{align}
    and 
       \begin{align}\label{eq_linear-conj-2}
       \sum_{y' \colon y\to y'\in \cR}\k_{y\to y'}(y' - y) \cdot (y' - y)^{\intercal} = \sum_{y' \colon y\to y'\in \cR'} \beta_{y\to y'} G (y' - y) \cdot (y' - y)^{\intercal} G^{\intercal}.  
       \end{align}
      Then $(\cN, \k')$ and $(\cN{{'}}, \k')$ are linearly conjugated w.r.t. their SDEs, where for all $y \to y' \in \cR'$,
       \begin{equation}\label{rates_conjugated_RN}
           \k'_{y\to y'} \coloneqq \beta_{y\to y'} c^y = \beta_{y\to y'}  \prod_{j=1}^n c_j^{y_{j}}.
       \end{equation}
       
\end{proposition}

\begin{remark}
Note that if we impose $G=Id$ the identity map, two different RNs $(\cN, \k)$ and $(\cN',\k')$ are linearly conjugated with scaling matrix $G = Id$ only if $\cN$ and $\cN'$ are confoundable w.r.t. their  SDEs.
\end{remark}

\begin{example}
   Consider the RNs:
$$\mathcal{N}:\quad S_1\xrightarrow{\alpha}  3S_1 \quad \text{and  } \quad\mathcal{N'}: \quad S_1\xrightarrow{\alpha'}  2S_1 .$$

Following Theorem \ref{weak_react_id}
these RNs  are not confoundable w.r.t. their SDEs
 since we cannot find $\alpha>0$ and $\alpha'>0$ such that:
\begin{equation*}
    2\alpha =  \alpha' \quad \text{and  } \quad  4\alpha =  \alpha' .
\end{equation*}
In order for ($\mathcal{N},\alpha$) and ($\mathcal{N'},\alpha'$) to be linearly-conjugated, we need to find $\alpha>0, \alpha'>0$ and $c_1>0$ such that \ref{eq_linear-conj-1} and \ref{eq_linear-conj-2} are verified i.e.: 
\begin{equation*}
    2\alpha =  c_1 \alpha' \quad \text{and  } \quad  4\alpha =  c_1^2 \alpha' .
\end{equation*}
It is verified for $\alpha = \alpha'$ and $c_1=2$. It follows that ($\mathcal{N},\alpha$) and ($\mathcal{N'},\alpha'$) are linearly-conjugated. 
\end{example}
Similarly as for reaction-identifiability and confoundability ({see Section~\ref{identif_generator}}) , one can easily obtain the following lemma.
\begin{lemma}

If two RNs are not linearly-conjugated w.r.t. their ODEs (see Definitions 3.1 and 3.2 of \cite{johnston2011linear}), then they are not linearly-conjugated w.r.t. their SDEs.
\end{lemma}

\section{Investigation of some special classes of RNs concerning identifiability}
\subsection{Special cases of reaction-identifiable RNs}
\begin{enumerate}

    \item Consider RNs, which only consist of the following types of reactions: 

$$S_i\rightarrow  S_j, \quad  S_i\rightarrow \emptyset, \quad \emptyset \rightarrow S_i ,$$
with chemical species $S_i$, $i=1,...,n$. 

For each complex the associated reaction vectors are linearly independent, and it follows from \cite[Theorem 3.2]{ident_CRNs}, that their corresponding RNs always have uniquely identifiable rate constants w.r.t. their ODEs and thus their SDEs.
This property extends easily to \textbf{k-Unary Chemical Reaction Networks} \cite{laurence2025analysis}, with reactions of the types: $$k_i S_i\rightarrow k_j S_j, \quad  k_i S_i\rightarrow \emptyset, \quad \emptyset \rightarrow k_i S_i ,$$
for $1 \leq i \neq j \leq N$, where $k = (k_i)_{1\leq i \leq n} \in \mathbb{Z}^n_{>0}$.

\item Simple \textbf{RNs with only one species}, have also uniquely identifiable rate constants w.r.t. the  SDE,
 if and only if each source complex is a reactant in at most two reactions. Otherwise the associated set of vectors in (\ref{generalized_reaction_vectors}) will be constituted of three or more vectors in $\mathbb{R}^2$, and will then be linearly dependent.

\item If each reaction in a reaction network has a distinct source complex, then the corresponding reaction vectors are automatically linearly independent, which in turn guarantees identifiability w.r.t. the ODE.
\end{enumerate}

\subsection{Counterexample}
Given that RNs with complexes of molecularity of at most one are uniquely identifiable w.r.t. their SDEs,
 one might expect the same 
for RNs with \textbf{reactants of molecularity at most one}. The following examples show this is not the case.

\begin{example}\label{non_ident_generator_multi_dim}
    Consider the following RN:
        
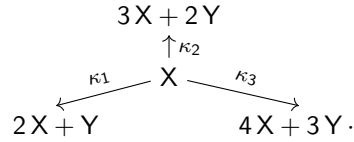
\begin{center}
\begin{tikzpicture}
  \matrix (m) [matrix of math nodes, row sep=1ex]{
    & \ce{3X + 2Y}  &\\
    &  & & & \\
    & \ce{X} &  \\
    \ce{2X + Y} & & \ce{4X + 3Y.}\\
  };
  \draw[->] (m-3-2) -- (m-1-2.south)
    node[right,pos=0.5,sloped,font=\footnotesize,rotate=270] {\ce{\kappa_2}};

  \draw[->] (m-3-2) -- (m-4-1.north)
    node[above,pos=0.5,sloped,font=\footnotesize] {\ce{\kappa_1}};
  \draw[->] (m-3-2) -- (m-4-3.north)
    node[above,pos=0.5,sloped,font=\footnotesize] {\ce{\kappa_3}};

\end{tikzpicture}
\end{center}

One can verify that the two sets of rate constant vectors $\kappa=(\kappa_1,\kappa_2,\kappa_3) = (2,7,5)$ and $\kappa' = (\kappa_1',\kappa_2',\kappa_3') = (5,4,6)$ give the same SDE. Hence this RN has not uniquely identifiable rate constants w.r.t. its SDE.

Note that if we remove the third chemical reaction namely  $X\xrightarrow{\kappa_3} 4X+3Y$, the set of vectors in (\ref{generalized_reaction_vectors}) becomes linearly independent, which makes the RN having uniquely identifiable rate constants w.r.t. the ODE and then also w.r.t. the SDE.
\end{example}

\section{Identifiability of RNs given generators}
\label{identif_generator}
The proof of our main result, Theorem \ref{weak_react_id}, relies on establishing identifiability of RNs given generators. In this section, we introduce the formal notion of identifiability of RNs given their generators and present several supporting results that culminate in the proof of Theorem \ref{weak_react_id}.

\subsection{More on SDEs}\label{SDEs_theory}
We start with introducing key concepts of SDEs relevant to this paper and refer the reader to e.g., \cite{ikeda1989stochastic,oksendal2013stochastic} for a more systematic study of the topic. 

In general, the SDE \eqref{Ito_sde} may not admit a unique solution for all $t > 0$. One issue arises from the fact that the diffusion coefficient  $\sigma (x) = \sqrt{B (x)}$ may fail to be Lipschitz continuous at the boundary of the positive orthant. In the one-dimensional case, this problem has been addressed in the seminal work \cite{yamada1971uniqueness},  which establishes the strong uniqueness for an SDE with an $\alpha$-H\"{o}lder continuous diffusion coefficient when $\alpha \geq 1/2$. However, in higher dimensions, where the equation under consideration is likely to reside, uniqueness is no longer guaranteed, as demonstrated in \cite{swart20012}. 

One possible remedy is to restrict the initial condition to lie strictly within the positive orthant. Even then, a second issue remains, that is the potential explosion of solutions in finite time. Since $\sigma (x)$ exhibits polynomial growth as $|x| \to \infty$, it may blow up in finite time. This phenomenon is well understood in the one-dimensional setting \cite{feller1954diffusion}, and has been generalised to higher dimensions in \cite{khasminskii1960ergodic}. We also note that the issue of (non-)explosion of RNs modelled by CTMCs has been investigated in \cite{non-expl}.

Nevertheless, both the drift  $A$ and the diffusion coefficient  $\sigma$ are Lipschitz continuous on compact subsets of the positive orthant and hence the following result holds, see, e.g., \cite{ikeda1989stochastic}. We begin with the standard definition of solutions to SDEs.

\begin{definition}
  A stochastic process $X = (X_t)_{t \geq 0}$ is called 

  \begin{enumerate}
    \item a \textbf{strong solution} to \eqref{Ito_sde}, if it is adapted to the filtration generated by the Brownian motion $W = (W_t)_{t \geq 0}$ in \eqref{Ito_sde}, and satisfies the following integral equation:
  \[
    X_t - X_0 = \int_0^t A(X_s) ds + \int_0^t \sigma (X_s) d W_s,
  \]
  where the stochastic integral is interpreted in the It\^o sense. A strong solution is called \textbf{unique}, if, whenever $X$ and $X'$ are any two strong solutions to \eqref{Ito_sde} such that $X_0 = X'_0$ a.s., then $X_t = X_t'$ for all $t \geq 0$, a.s. 

  \item a \textbf{weak solution} to \eqref{Ito_sde}, if there exists a probability space and a Brownian motion $\widetilde{W} = (\widetilde{W}_t)_{t \geq 0}$ with $\widetilde{W}_0 = 0$, such that $X$ is adapted to the filtration generated by $\widetilde{W}$ and 
  satisfies 
  \[
    X_t - X_0 = \int_0^t A(X_s) ds + \int_0^t \sigma (X_s) d \widetilde{W}_s,
  \]
  where the stochastic integral is interpreted in the It\^o sense. A weak solution is called \textbf{unique} (in law), if, whenever $X$ and $X'$ are any two weak solutions to \eqref{Ito_sde} such that $X_0 \overset{d}{=} X'_0$, then $X_t \overset{d}{=} X'_t$ for all $t \geq 0$. 
  \end{enumerate}

\end{definition}

{\begin{theorem}\label{uniqueness}
    If we let $U \subseteq \R_{>0}^n$ be a bounded open set and take an initial condition $x \in U$, then \eqref{Ito_sde} admits a unique strong solution up to the stopping time $\tau = \inf\{t > 0\colon X_t (x) \notin U\}$, and therefore also a unique weak solution.
\end{theorem}}
This is sufficient for our investigation of identifiability.
Moreover, the solution $(X_{t\wedge \tau}(x))_{t \geq 0}$ is a homogeneous Markov process. In the following, we introduce the concept of the \textbf{infinitesimal generator} for homogeneous Markov processes.
 \begin{definition}\cite[Definition 7.3.1]{oksendal2013stochastic} \label{def_generator-0}
Let $(X_t)_{t \geq 0}$ be a homogeneous Markov process with state space $D \subseteq \R^n$. The \textbf{infinitesimal generator} (or simply the generator) of $X$ is an operator $L$ on $C_c^{\infty} (D)$, such that for all $f \in C_c^{\infty} (D)$ and $x \in D$,
\begin{align}\label{def_generator}
  L f(x) \coloneqq \lim_{t \to 0} \frac{\E[f(X_t(x))] - f(x)}{t}.
\end{align}
For a general Markov process, the generator may not always exist—additional properties, such as strong continuity of the process, are required to ensure that the limit in \eqref{def_generator} exists. In this paper, we focus on the Langevin approximation (LA) for reaction networks (RNs), where the generator exists and takes the form in \eqref{generator}.
\end{definition}

Then, \cite[Proposition 1.7 in Chapter 4]{ethier1986markov} asserts the following theorem.
\begin{theorem}
  Let $(X_t)_{t \geq 0}$ be a Markov process with generator $L$. Then, it solves the \textbf{martingale problem} for $L$, namely, for every $f \in C_c^{\infty} (\R^d)$,
  \[
    M_t \coloneqq f(X_t) - \int_0^t L f(X_s) ds
  \]
  is a martingale adapted to the filtration generated by $(X_t)_{t \geq 0}$.
\end{theorem}

If the Markov process $(X_t)_{t \geq 0}$ is the solution to an It\^{o} equation \eqref{Ito_sde}, with locally Lipschitz continuous coefficients $A$ and $\sigma$ on $\R_{> 0}^n$ then, its infinitesimal generator can be expressed as follows
\begin{equation}\label{generator}
    Lf(x) = \langle A(x), \nabla \rangle f{(x)}+{\frac{1}{2}} (\nabla^{\intercal} \cdot B (x) \cdot \nabla) f{(x)} = \sum_{i = 1}^n A_i(x)\frac{\partial}{\partial x_i}f(x)+\frac{1}{2}\sum_{i,j = 1}^n B_{i,j}(x)\frac{\partial^2}{\partial x_i\partial x_j}f(x),
\end{equation}
for all test functions $f\in C^{\infty}_c(\R_{> 0}^n)$, cf. \cite[Theorem 7.3.3]{oksendal2013stochastic}. 

\begin{remark}\label{gen_unique_AB} It is clear that drift and diffusion $A,B$ determine a generator of the form \eqref{generator}. Note that also the generator \eqref{generator} uniquely determines $A,B$, where we recall that $B$ is symmetric. This can be seen, e.g, by considering test functions in $C_c^{\infty} (\R_{> 0}^n)$ that equal $x_i$ or $x_ix_j$ on a bounded open set inside $\R_{> 0}^n$ \cite{ethier1986markov}.
\end{remark}

The following theorem, concerning the equivalence of martingale problems and SDEs, was first established in \cite{stroock1972support} and further studied in \cite{Kurtz2011EquivalenceOS}. The equivalence is given in terms of existence of SDEs through a weak solution, i.e., a filtered probability space and stochastic process that satisfies the SDE in its integral form \cite[Section 5.5]{oksendal2013stochastic}.

\begin{theorem}\label{thm_eqv-sde-mp}
  A stochastic process $(X_t)_{t \geq 0}$ is a solution to the martingale problem for $L$ of the form \eqref{generator}, if and only if it is a weak solution to \eqref{Ito_sde}.
\end{theorem}

    \begin{example}\label{1d_non_ident_generator}
In order to exemplify notation, we again consider Example \ref{1d_non_ident_generator_intro}:
$$2S\xleftarrow{1}\emptyset\xrightarrow{4}S\xrightarrow{1} \emptyset\xrightarrow{2}3S , $$
$$2S\xleftarrow{4}\emptyset\xrightarrow{1}S \xrightarrow{1} \emptyset\xrightarrow{1}3S .$$
The stoichiometric matrix is a row vector: $\Gamma = (
    2,1,-1,3)$, which simplifies the calculation of both the drift vector and the diffusion matrix, that are
$A(x)=12-x$, and $ B(x)=x+26$ for all $x \in \R_{>0}$. Therefore, the LA for the RN is
\[
  d X_t = (12 - X_t) dt + \sqrt{X_t + 26} d W_t; 
\]
and the
 infinitesimal generator associated to the RN is
\[
  L f(x) = (12 - x) \frac{d}{d x} f(x) + \frac{1}{2} (x + 26) \frac{d^2}{d x^2} f(x).
\]
\end{example}

\subsection{Reaction-identifiability and confoundability given generator}
\begin{definition}
Let $\cN = (\cS, \cC, \cR)$ and $\cN' = (\cS, \cC', \cR')$ be two RNs with the same set of species, where $\cR \neq \cR'$.
\begin{itemize}
    \item $\cN$ is \textbf{reaction-identifiable} w.r.t. its generator, if and only if for any distinct rate constant vectors $\k , \k' \in\R^d_{>0}$, the corresponding generators of their LAs are different.

    \item $\cN$ and $\cN'$ are \textbf{confoundable} w.r.t. their generators, if there do exist rate constant vectors $\k \in \R^d_{>0}$ and $\k'\in\R^{d'}_{>0}$ such that the corresponding generators of their LAs are identical, where $d$ and $d'$ denote the numbers of reactions in $\cN$ and $\cN'$, respectively. Otherwise, we say that $\cN$ and $\cN'$ are \textbf{unconfoundable} w.r.t. their generators.

\end{itemize}
\end{definition}

In analogy to  \cite{ident_CRNs}, we have the following result, whose proof is deferred to Subsection \ref{prf_theorem1}.
\begin{theorem}\label{theorem1}
Let $\cN = (\cS, \cC, \cR)$ be an RN. The followings are equivalent.
\begin{enumerate}
\item $\cN$ is reaction-identifiable w.r.t. its generator.

\item There is a non-empty bounded open set $U\subseteq\R^n_{>0}$, such that for all rate constant vector{s} $\kappa \neq \kappa'\in\R^d_{>0}$ the corresponding generators are different on $C^{\infty}_c(U)$.

\item For every source complex $y\in \cC$, the set of vectors in \eqref{generalized_reaction_vectors}
is linearly independent.

\end{enumerate}

\end{theorem}

\begin{remark}\label{rmk-theorem1}
   If two RNs $\cN$ and $\cN'$ are unconfoundable w.r.t. their generators, then an analogous of property
(2) in Theorem \ref{theorem1} holds for all rate constant vectors $\kappa \in \mathbb{R}^d_{>0}$ and $\kappa' \in \mathbb{R}^{d'}_{>0}$; see Subsection \ref{prf_theorem1}.
\end{remark}

\begin{lemma}\label{conseq_lemma}
{If an RN is reaction-identifiable  w.r.t. its ODE, then this RN is also reaction-identifiable w.r.t. its generator.}
\end{lemma}
{\begin{proof}
 Suppose an RN is reaction-identifiable  w.r.t. its ODE. Then, according to \cite[Theorem 3.2]{ident_CRNs}, for any source complex $y \in \cC$, the set $\{y' - y \colon y \to y' \in \cR\}$ is linearly independent. This, in turn, implies that the set of vectors in \eqref{generalized_reaction_vectors} is also linearly independent. Therefore, the RN is also reaction-identifiable w.r.t. its generator by Theorem \ref{theorem1}.
\end{proof}
}

However, the converse is not necessarily true, as demonstrated in the following example.

\begin{example}\label{ex_ODE_yes_SDE_no}
Consider the following birth and death reaction network:
$$\emptyset\xleftarrow{\kappa_1} S\xrightarrow{\kappa_2} 2S,$$
which is used to describe the evolution of cancer cells apart from other applications in biology \cite{10.1093/bib/bbk006}.

There is only one source complex $y = S$ and we have that the set of vectors defined in (\ref{generalized_reaction_vectors}) i.e.$ \{ (1,1)^{\intercal},  (-1,1)^{\intercal}\}$ are linearly independent. Therefore, according to Theorem \ref{theorem1}, this RN is reaction-identifiable w.r.t. its generator. 
However, if we consider $(\k_1,\k_2) = (1.5,1)$ and $(\k_1',\k_2') = (2,1.5)$, then both reaction rates yield the same ODE system:
\[
  \dot{x} = x-1.5x = 1.5x-2x=-0.5x.
  \]
It follows that the RN is not reaction-identifiable w.r.t. its ODE.
\end{example}
\begin{remark}
  In Example \ref{ex_ODE_yes_SDE_no}, $S$ is the only source complex in the RN, which has deficiency one. In fact, it is not hard to verify that a reaction network with a single source complex is reaction-identifiable w.r.t. its ODE if and only if it has deficiency zero. However, as demonstrated in Example \ref{ex_ODE_yes_SDE_no}, having zero deficiency is only a sufficient but not necessary condition for reaction-identifiability w.r.t. its generator.
\end{remark}

If an RN is not reaction-identifiable w.r.t. its generator, then it remains non-reaction-identifiable even with adding additional reactions.

\begin{lemma}\label{cor_subnetwork}
  Suppose an RN $\mathcal{N}$ is not reaction-identifiable w.r.t. its generator. Then, for every reaction network $\mathcal{N}'$ such that $\mathcal{N}$ is a subnetwork of $\cN'$, $\cN'$ is also not reaction-identifiable w.r.t. its generator. 
\end{lemma}
{\begin{proof}
    $\mathcal{N}$ is not reaction-identifiable w.r.t. its generator so there exists two vectors of rate constants $\kappa^1_{\mathcal{N}}$ and $\k^2_{\mathcal{N}}$ giving the same generator.  For any $\mathcal{N}'$ with $d' \geq d$ reactions such that $\cN$ is a subnetwork of $\cN'$, we consider the rate constant vectors $\kappa^1_{\cN'} = (\kappa^1_{\cN},a)$ and $\kappa^2_{\mathcal{N}'}=(\kappa^2_{\mathcal{N}},a)$ where $a \in \R^{d-d'}$. These two vectors give the same generator and thus $\cN'$ is not reaction-identifiable w.r.t. its generator. 
\end{proof}}

\subsection{{Confoundability} of RNs given generators}
In the previous subsection, we consider a given RN and its generator.

In this subsection, we investigate the confoundability of RNs given the generator.

\begin{theorem}\label{confoundability}
Two RNs $\cN = (\cS, \cC, \cR)$ and $\cN' = (\cS, \cC', \cR')$ are confoundable w.r.t. their generators if and only if they have the same source complexes and for every source complex $y\in\cC$, ${\rm Cone}_\cR(y) \cap {\rm Cone}_{\cR'}(y) \neq \emptyset$, where ${\rm Cone}_\cR(y) $ is given by \eqref{def_cone}.
\end{theorem}
The proof of Theorem \ref{confoundability} can be found in Subsection \ref{prf_confoundability}.

\begin{example}\label{example_cofoundability} 
Consider the following two different RNs:
$$S\xleftarrow{5}\emptyset\xrightarrow{1}4S,  \quad S\xrightarrow{1}\emptyset , 
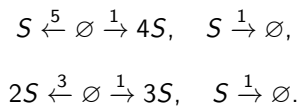
$$
$$2S\xleftarrow{3}\emptyset\xrightarrow{1}3S, \quad S\xrightarrow{1}\emptyset .$$
Again both RNs give the same ODE
$$\frac{ds}{dt}=9-s.$$
Despite the RNs being different, there exists two sets of reaction rates for which the ODEs are the same. Hence the RN structure cannot be distinguished from the ODE alone. This is called (ODE)-confoundable in \cite{craciun2}. Furthermore both RNs also have their SDEs given by the same drift vector and diffusion matrix
$$A(s)=9-s,\quad B(s)=s+21.$$
These RNs are thus confoundable w.r.t. their generator. 
As the corresponding SDE is regular enough, 
 we conclude that these RNs are confoundable w.r.t. the SDE. 
\end{example}

\begin{remark}\label{remark_identical_source_complex}

    As pointed out in \cite{szederkenyi2009comment}, the confoundability of two RNs w.r.t. their ODEs does not necessarily require their source complexes to be identical. A counterexample in \cite{szederkenyi2009comment} is constructed by choosing special rate constants for reactions with the additional source complex, such that its contribution to the corresponding ODE vanishes. However, this does not apply in our case, as the contribution of each source complex to the diffusion matrix is nonzero and non-negative definite. For further details, we refer the reader to the proof in Subsection  \ref{prf_confoundability}.

\end{remark}

Similar to Lemma \ref{conseq_lemma}, we state the following lemma.
\begin{lemma}\label{l.unconfoundability}
If two RN are unconfoundable w.r.t. their ODEs, then they are unconfoundable w.r.t. their generators.
\end{lemma}
Not surprisingly, the converse of Lemma~\ref{l.unconfoundability} does not hold, as demonstrated in Remark~\ref{remark_identical_source_complex}.

By the definition of the generator, and in a manner similar to Lemma \ref{cor_subnetwork}, the following extension result holds.
\begin{lemma}\label{confoundability_added_reactions}{    Let $\cN_1 = (\cS, \cC_1, \cR_1)$ and $\cN_2 = (\cS, \cC_2, \cR_2)$}
 be two RNs confoundable w.r.t. their generators, and let $r\not\in\mathcal{R}_1,r\not\in\mathcal{R}_2$.  Then also the RNs obtained by adding reaction $r$ to both $\mathcal{R}_1$ and $\mathcal{R}_2$, are confoundable w.r.t. their generators.
\end{lemma}

\section{Conclusion}
This work rigorously examines the identifiability of SDEs derived from biochemical reaction networks under mass-action kinetics. Focusing on the Langevin approximation, we establish necessary and sufficient conditions for the identifiability of rate constants and reaction network structures based on the generator of the corresponding SDEs. We demonstrate that identifiability in the stochastic setting is finer than in the deterministic model, and that distinct reaction networks can generate indistinguishable stochastic dynamics. The results have direct implications for parameter inference and model selection in systems biology, offering tools to detect when models or parameters are indistinguishable from data. Future directions include extending the analysis to reflected diffusions under constrained Langevin approximations, or using different kinetics.

Conversely, our identifiability analysis was carried out under the assumption that the complete reaction-network is known and all species are observed. In practice, however, only a subset of species is typically observed, and it is therefore of practical interest to establish identifiability under partial observability. This problem is relatively well studied in the ODE setting for monomolecular reaction networks \cite{linear_RN_ODE}, and analogous questions can be addressed for SDE models as well. Moreover, observations are usually available only at discrete time points, where transition densities are rarely available in closed form (except in special cases), and, e.g., maximum-likelihood inference is both numerically and theoretically delicate \cite{SDE_inference}. Finally, while identifiability is a necessary condition for parameter estimation, it is generally not sufficient.

\section{Proofs}

\subsection{Proof of Theorem \ref{weak_react_id}} \label{prf_weak_react_id}
With Theorems \ref{theorem1} and \ref{confoundability} in hand, the validity of Theorem \ref{weak_react_id} reduces to establishing the following proposition.

\begin{proposition}
Let $\cN = (\cS, \cC, \cR)$ and $\cN' = (\cS, \cC', \cR')$ be two RNs with the same set of species, where $\cR \neq \cR'$.

\begin{enumerate}
\item $\cN$ is reaction-identifiable w.r.t. its SDE if and only if $\cN$ is reaction-identifiable w.r.t. its generator.

\item $\cN$ and $\cN'$ are unconfoundable w.r.t. their SDEs if and only if they are unconfoundable w.r.t. their generators.
\end{enumerate}

\end{proposition}
\begin{proof}
(1) $\implies$:~
Suppose that $\cN$ is not reaction-identifiable w.r.t. its generator. Then there exists distinct rate constant vectors $\kappa$ and $\kappa'$  in  $\R^{d}_{>0}$ associated with the same generator on some non-empty bounded open subset $U$. By Theorem \ref{thm_eqv-sde-mp} \textemdash the equivalence of the martingale problem and the SDE \textemdash we know that the stochastic processes
associated to $\kappa$ and $\kappa'$
are weak solutions to the same SDE. Additionally, using Theorem~\ref{uniqueness}, solutions to this SDE are unique in law up to the stopping time $\tau = \inf\{t > 0, X_{t} \notin U\}$. That is, $\kappa$ and $\kappa'$ provide two RNs with the identical distribution when starting from the same initial state. Hence $\cN$ is not reaction-identifiable w.r.t. its SDE. In other words, $\cN$ is reaction-identifiable w.r.t. its generator whenever it is reaction-identifiable w.r.t. its SDE.

$\impliedby$:~
 Suppose that $\cN$ isn't reaction-identifiable w.r.t. its SDE, then there exists $\kappa \neq  \kappa' \in \R^{d}_{>0}$ such that the two stochastic processes $(X_t)_{t \geq 0}$ and $(X'_t)_{t \geq 0}$ associated to $(\cN,\kappa)$ and $(\cN,\kappa')$ have the same distribution, when stopped upon reaching the boundary of $U$. Then, the corresponding generators coincide,
as by Definition \ref{def_generator-0},
 \begin{align}\label{eq_L=L'}
  L f(x) = \lim_{t \downarrow 0}\frac{\mathbb{E} [f(X_t (x))] - f(x)}{t} = \lim_{t \downarrow 0}\frac{\mathbb{E} [f(X'_t (x))] - f(x)}{t} = L' f(x),
\end{align}
for any $x \in U$, and  test function $f \in C_c^{\infty} (U)$. This ensures that $\cN$ isn't reaction-identifiable w.r.t. its generator. \medskip

 (2) $\implies$:~
Suppose that $\cN$ and $\cN'$ are confoundable w.r.t. their generators. Then there exists $\kappa \in \R^{d}_{>0}$ and $\kappa' \in \R^{d'}_{>0}$, such that $(\cN,\kappa)$ and $(\cN',\kappa')$ have same generator on some non-empty bounded open subset $U$. By Theorem \ref{thm_eqv-sde-mp}\textemdash the equivalence of the martingale problem and the SDE \textemdash  we know that the stochastic processes
associated to $(\cN,\kappa)$ and $(\cN',\kappa')$
are weak solutions to the same SDE, and therefore share the same law up to the stopping time $\tau = \inf\{t > 0, X_{t} \notin U\}$ due to Theorem~\ref{uniqueness}. 
This proves that $\cN$ and $\cN'$ are confoundable w.r.t. their SDEs. Therefore, $\cN$ and $\cN'$ are unconfoundable w.r.t. their generators, whenever, they are unconfoundable w.r.t. their SDEs.

 $\impliedby$:~
 Suppose that $\cN$ and $\cN'$ are confoundable w.r.t. their SDEs, then there exists $\kappa \in \R^{d}_{>0}$ and $\kappa' \in \R^{d'}_{>0}$ such that the two stochastic processes $(X_t)_{t \geq 0}$ and $(X'_t)_{t \geq 0}$ associated to $(\cN,\kappa)$ and $(\cN',\kappa')$ have same distribution, when stopped upon reaching the boundary of $U$. Thus equation \eqref{eq_L=L'} holds
for any $x \in U$, and  test function $f \in C_c^{\infty} (U)$. This ensures that $\cN$ and $\cN'$ are confoundable w.r.t. their generators. The proof of this proposition is complete.
\end{proof}

\subsection{Proof of Theorem \ref{theorem1}} \label{prf_theorem1} 
We will need the following elementary fact.

\begin{lemma}\label{lem:monomial_independence}
Let $Y\subset \mathbb Z_{\ge 0}^n$ be a finite set and let $v_y\in\mathbb R^m$ be vectors for $y\in Y$. If
\[
\sum_{y\in Y} v_y\,x^y = 0 \qquad \text{for all } x\in U,\quad U\subseteq\mathbb R^n\text{ open},
\]
then $v_y=0$ for all $y\in Y$.
\end{lemma}
\begin{proof}
Each coordinate of $\sum_{y\in Y} v_y x^y$ is a polynomial in $(x_1,\dots,x_n)$ that vanishes on the open set $U\subseteq \mathbb R^n$. Hence it is the zero polynomial and all coefficients vanish \cite{mityagin2015zerosetrealanalytic}.
\end{proof}

\begin{proof}[Proof of Theorem~\ref{theorem1}]$(1)\implies (2)$: Let $\cN$ be an RN such that $(2)$ fails. Then, there exist rate constant vectors $\kappa \neq \kappa'$, whose corresponding generators coincide on a nonempty open set $U$. Recall that the generators take the form \eqref{generator}. The fact that two generators coincide on $C_c^{\infty} (U)$ is equivalent to their coefficients $A_i$ and $B_i$, which are polynomials, coinciding on $U$. Since a nonzero polynomial cannot vanish on a nonempty open set, it follows that the coefficients (and hence the generators) coincide on all of $\R^n_{>0}$ (resp. on $C^\infty_c(\mathbb{R}^n_{>0})$), see Lemma \ref{lem:monomial_independence}. As a consequence, $\cN$ is not identifiable w.r.t. its generator; that is $(1)$ fails. This proves the implication $(1) \implies (2)$.

Notice that in the above argument, we do not actually require the underlying RNs with rate constant vectors $\kappa$ and $\kappa'$ to share the same network structure. Therefore, the proof still justifies Remark \ref{rmk-theorem1} without any modification.

$(2) \implies (3)$: Assume that on an open subset $U$ as given the generators are different for all $\kappa \neq \kappa'$, but that there is a source complex $y \in \cC$ such that the set of vectors 
$$\left\{\left(y' - y,(y' - y)\cdot (y' - y)^{\intercal}\right) \big| \, y\to y'\in\cR\right\}$$
is linearly dependent. Then there exist some real numbers $\alpha_{y\to y'}$, not all zero  such that
\begin{align*}
\left(\sum_{y' \colon y \to y'\in\cR} \alpha_{y\to y'}(y' - y),\sum_{y' \colon y \to y'\in\cR} \alpha_{y\to y'} (y' - y)\cdot (y' - y)^{\intercal} \right)  
= 0.
\end{align*}
Hence we can choose rate constant vectors $\k,\k' \in \R^{d}_{>0}$ such that $\k_{y \to y'}- \k'_{y \to y'} = \alpha_{y \to y'}$ for all coefficients of reactions $y \to y'\in\cR$, and $\k_{R} = \k'_{R} = 1$ for other reactions $R \in\cR$. Then, recalling \eqref{drift_vector}, \eqref{diffusion_matrix}, and \eqref{generator}, for these rate constant vectors,
with $L$ and $L'$ denoting the generator with $k$ and $k'$ respectively,
\begin{align*}
 L f(x) - L' f(x) = & \sum_{y\to y'\in \cR} \left\langle (y' - y)(\k_{y\to y'}-\k'_{y\to y'})x^y, \nabla\right\rangle f \\ 
& + \frac{1}{2} \sum_{y\to y'\in \cR} \left(\nabla^{\intercal} \cdot (y' - y)\cdot(y' - y)^{\intercal}(\k_{y\to y'}-\k'_{y\to y'}) x^y \cdot \nabla \right) f  = 0.
\end{align*}
This contradicts with our assumption, and thus verifies $(2) \implies (3)$.

$(3) \implies (1)$: 
 Assume that for all source complexes $y \in \cC$ the set of vectors 
$$\left\{\left(y' - y,(y' - y)\cdot (y' - y)^{\intercal}\right) \big| \, y\to y'\in\cR \right\}$$
is linearly independent. Let $\k,\k'\in\R^d_{>0}$ such that $L f(x) - L' f(x) = 0$, for all test functions $f\in C^{\infty}_c(\R_{> 0}^n)$ and $x \in R_{> 0}^n$. In other words,

$$\sum_{y \in \cC} \left(\sum_{y' \colon y\to y'\in \cR}(y' - y)(\k_{y\to y'} - \k'_{y\to y'})x^y,\sum_{y' \colon y\to y'\in \cR}(y' - y)\cdot(y' - y)^{\intercal}(\k_{y\to y'}-k'_{y\to y'})x^y \right)= 0,$$
where the right hand side is the zero-function $\R^n\to \R^{n}\times\R^n \otimes \R^n$. This yields that
$$\sum_{y\in \cC}x^y\sum_{y' \colon  y\to y'\in \cR} (\k_{y\to y'}-\k'_{y\to y'}) \left((y' - y),(y' - y)\cdot(y' - y)^{\intercal}\right)= 0.$$
Since these vectors are polynomials in $x$ vanishing on $\R^n_{>0}$, their polynomial coefficients must be zero. Therefore, for each source complex $y\in\cC$,
$$
\sum_{y' \colon y\to y'\in \cR}(\k_{y\to y'}-\k'_{y\to y'}) \left((y' - y),(y' - y)\cdot(y' - y)^{\intercal}\right)= 0 .
$$
By the linear independence, we get that $\k_{y\to y'}-\k'_{y\to y'}=0$. This concludes the proof of $(3) \implies (1)$. The proof of this theorem is complete.
\end{proof}

\subsection{Proof of Theorem \ref{confoundability}}\label{prf_confoundability}
{
First, note that the condition ${\rm Cone}_{\cR}(y)\cap {\rm Cone}_{\cR'}(y)\neq\emptyset$ for all $y\in \cC\cup\cC'$ implies that $\cN$ and $\cN'$ have the same source complexes: if $y$ is not a source complex of $\cR$, then ${\rm Cone}_{\cR}(y)=\{0\}$ by convention on the empty sum (and similarly for $\cR'$), but the coefficients of the coordinates in ${\rm Cone}_{\cR'}(y)\subset\R^n \otimes \R^n$ from \eqref{def_cone} are positive, so the intersection is empty. That is the difference to the ODE case of the proof of \cite[Theorem 4.4]{ident_CRNs} that was remarked in \cite{szederkenyi2009comment}.Thus it suffices to prove the equivalence between generator confoundability and the cone condition.

\medskip
 $\implies$:~
Assume $\cN$ and $\cN'$ are confoundable with respect to their generators. Then there exist rate constant vectors $\kappa\in\mathbb R^{|\cR|}_{>0}$ and $\kappa'\in\mathbb R^{|\cR'|}_{>0}$ such that, for all $x\in\mathbb R^n_{>0}$,
\begin{equation}\label{eq:gen_conf_poly}
\begin{dcases}
\displaystyle
\sum_{y\in \cC\cup\cC'}
\Bigg(
\sum_{\substack{y':\,y\to y'\in\cR}} (y'-y)\,\kappa_{y\to y'}
-
\sum_{\substack{y':\,y\to y'\in\cR'}} (y'-y)\,\kappa'_{y\to y'}
\Bigg)x^y = 0,\\[1.2ex]
\displaystyle
\sum_{y\in \cC\cup\cC'}
\Bigg(
\sum_{\substack{y':\,y\to y'\in\cR}} (y'-y)\cdot(y'-y)^{\intercal}\,\kappa_{y\to y'}
-
\sum_{\substack{y':\,y\to y'\in\cR'}} (y'-y)\cdot(y'-y)^{\intercal}\,\kappa'_{y\to y'}
\Bigg)x^y = 0.
\end{dcases}
\end{equation}
Apply Lemma~\ref{lem:monomial_independence} to the first (with $m=n$) and the second line simultaneously. We obtain, for every $y\in \cC\cup\cC'$,
\begin{equation}\label{eq:coeffwise_equalities}
\begin{dcases}
\displaystyle
\sum_{\substack{y':\,y\to y'\in\cR}} (y'-y)\,\kappa_{y\to y'}
=
\sum_{\substack{y':\,y\to y'\in\cR'}} (y'-y)\,\kappa'_{y\to y'},\\[1.2ex]
\displaystyle
\sum_{\substack{y':\,y\to y'\in\cR}} (y'-y)\cdot(y'-y)^{\intercal}\,\kappa_{y\to y'}
=
\sum_{\substack{y':\,y\to y'\in\cR'}} (y'-y)\cdot(y'-y)^{\intercal}\,\kappa'_{y\to y'}.
\end{dcases}
\end{equation}
Fix a source complex $y$. The pair determined by the left-hand sides of~\eqref{eq:coeffwise_equalities} belongs to ${\rm Cone}_{\cR}(y)$ by definition, while the equal pair determined by the right-hand sides belongs to ${\rm Cone}_{\cR'}(y)$. Hence
\[
{\rm Cone}_{\cR}(y)\cap {\rm Cone}_{\cR'}(y)\neq \emptyset
\qquad \text{for every source complex } y.
\]

\medskip
 $\impliedby$:~
Assume now that $\cN$ and $\cN'$ have the same source complexes and that for each source complex $y$,
${\rm Cone}_{\cR}(y)\cap {\rm Cone}_{\cR'}(y)\neq \emptyset$.
For each such $y$, choose an element of the intersection and corresponding positive rate subvectors
$\kappa^{(y)}$ (for reactions in $\cR$ with source $y$) and $\kappa'^{(y)}$ (for reactions in $\cR'$ with source $y$) such that
\begin{equation}\label{eq:local_choice}
\begin{dcases}
\displaystyle
\sum_{\substack{y':\,y\to y'\in\cR}} (y'-y)\,\kappa^{(y)}_{y\to y'}
=
\sum_{\substack{y':\,y\to y'\in\cR'}} (y'-y)\,\kappa'^{(y)}_{y\to y'},\\[1.2ex]
\displaystyle
\sum_{\substack{y':\,y\to y'\in\cR}} (y'-y)\cdot(y'-y)^{\intercal}\,\kappa^{(y)}_{y\to y'}
=
\sum_{\substack{y':\,y\to y'\in\cR'}} (y'-y)\cdot(y'-y)^{\intercal}\,\kappa'^{(y)}_{y\to y'}.
\end{dcases}
\end{equation}
Define global rate constant vectors $\kappa\in\mathbb R^{|\cR|}_{>0}$ and $\kappa'\in\mathbb R^{|\cR'|}_{>0}$ by setting, for each reaction $y\to y'$,
\[
\kappa_{y\to y'} := \kappa^{(y)}_{y\to y'},
\qquad
\kappa'_{y\to y'} := \kappa'^{(y)}_{y\to y'}.
\]
This is well-defined because each reaction has a unique source complex $y$. Then~\eqref{eq:local_choice} implies~\eqref{eq:coeffwise_equalities} for every source complex $y$.
Multiplying~\eqref{eq:coeffwise_equalities} by $x^y$ and summing over $y$ yields~\eqref{eq:gen_conf_poly} for all $x\in\mathbb R^n_{>0}$.
Therefore $\cN$ and $\cN'$ are confoundable with respect to their generators.
}

\subsection{Proof of Theorem \ref{thm_linear-conj}} \label{prf_linear-conj}
The proof is similar in structure to that of Theorem \ref{confoundability}.
First, $(\cN, \k)$ and $(\cN', k')$ are linear conjugate w.r.t. their SDEs, if and only if there exists a linear mapping $h' (x) = G x$, such that, for any bounded open subset $U \subseteq \R_{>0}^n$ \eqref{linear_map} holds. Applying the same reasoning as in Theorem~\ref{confoundability}, we conclude that $\cN$ and $\cN'$ must have the same source complexes to be linear conjugate. On the other hand, \cite[Lemma 3.1]{johnston2011linear} asserts the matrix $G$ consists of at most positively scaling and reindexing coordinates. Additionally, $h' (X_{t \wedge \tau})$ is a Markov process solving the martingale problem for the generator
\begin{equation*}
    L^{h'} f(x) = \langle G A(x), \nabla \rangle f + (\nabla^{\intercal} \cdot G B G^{\intercal} (x) \cdot \nabla) f .
\end{equation*}
Thus, \eqref{linear_map} is equivalent to equations 
\begin{equation*}
    \begin{dcases}
     &  \left(\sum\limits_{y' \colon y\to y'\in \cR}(y' - y) \k_{y\to y'} - \sum\limits_{y' \colon y\to y'\in \cR'}G(y' - y)\k'_{y\to y'} \right) x^y = 0, \\
     & \left(\sum\limits_{y' \colon y\to y'\in \cR}  (y' - y)\cdot(y' - y)^{\intercal}  \k_{y\to y'}-\sum\limits_{y' \colon y\to y'\in \cR'}G(y' - y)\cdot(y' - y)^{\intercal}G^{\intercal} \k'_{y\to y'} \right) x^y = 0,
\end{dcases}
\end{equation*}
for every source complex $y$ and state $x \in U$, which holds if and only if $
    {\rm Cone}_{\cR} (y) \cap  {\rm Cone}^{G}_{\cR'} (y) \neq \emptyset$. This completes the proof of this theorem.
\qed

\subsection{Proof of Proposition \ref{prop_linear-conj-rate}}\label{prf_linear-conj-rate}

Let $U \subseteq \R_{> 0}^n$ be an arbitrary bounded open set. Let $X_{t \wedge \tau} (x)$ and $X_{t\wedge \tau'}' (z)$ be 
solutions to the LAs for $(\cN, \k)$ and $(\cN', \k')$, starting at $x \in U$ and $z = h(x) \coloneqq G^{-1} x \in h(U)$, respectively, up to stopping times given by \eqref{def_tau-tau'}. Thus, concerning \eqref{drift_vector}, \eqref{diffusion_matrix}, employing It\^{o}'s formula, and recalling that ${G}$ is a diagonal matrix, we can write
\begin{align*}
  f \big(h \big(X_{t\wedge \tau}(x)\big)\big) = & f(z) + \sum_{i = 1}^n  \int_0^{t\wedge \tau}   \sum_{r = 1}^{d} c_i^{-1} (y'_{i,r} - y_{i,r} ) \k_r X_{t\wedge \tau}(x)^{y_r} f^{(i)} \big(h \big(X_{t\wedge \tau}(x)\big)\big) \\ 
 & + \frac{1}{2}\sum_{i,j = 1}^n \int_0^{t\wedge \tau}  \sum_{r = 1}^d c_i^{-1} c_j^{-1} (y'_{i,r} - y_{i,r}) (y'_{j,r} - y_{j,r}) \k_{r} X_{t\wedge \tau}(x)^{y_r}  f^{(i, j)} \big(X_{t\wedge \tau}(x)\big)  + M_{t \wedge \tau}\\ 
 = &  \sum_{i = 1}^n  \sum_{y \in \cC} h \big(X_{t\wedge \tau}(x)\big)^{y} c^y \sum_{y' \colon y \to y'  \in \cR}  c_i^{-1} \k_{y \to y'} (y'_{i} - y_{i} )  \frac{\partial}{\partial z_i}  f^{(i)} \big(h \big(X_{t\wedge \tau}(x)\big)\big) \\ 
  & +\frac{1}{2}\sum_{i,j = 1}^n   \sum_{y \in \cC}  h \big(X_{t\wedge \tau}(x)\big)^y c^y  \sum_{y' \colon y \to y' \in \cR} c_i^{-1} c_j^{-1}  \k_{y \to y'} (y_{i}' - y_{i}) (y_{j}' - y_{j})   f^{(i, j)} \big(X_{t\wedge \tau}(x)\big)  + M_{t \wedge \tau},
\end{align*}
where $f^{(i)} (x) \coloneqq \frac{\partial}{\partial x_i} f(x)$, $f^{(i, j)} (x) \coloneqq \frac{\partial^2}{\partial x_i \partial x_j} f(x)$, and $M_{t \wedge \tau}$ is a martingale. Therefore, the Markov process $h (X_{t \wedge \tau} (x)) = G^{-1} X_{t \wedge \tau} (x)$ solves the martingale problem for $L^h$ given by
\begin{align*}
  L^h f (z) 
  = & \sum_{i = 1}^n  \sum_{y \in \cC} z^{y} c^y \sum_{y' \colon y \to y'  \in \cR}  c_i^{-1} \k_{y \to y'} (y'_{i} - y_{i} )  \frac{\partial}{\partial z_i} f(z) \\ 
  & +\frac{1}{2}\sum_{i,j = 1}^n   \sum_{y \in \cC}  z^y c^y  \sum_{y' \colon y \to y' \in \cR} c_i^{-1} c_j^{-1}  \k_{y \to y'} (y_{i}' - y_{i}) (y_{j}' - y_{j})  \frac{\partial^2}{\partial z_i\partial z_j}f(z),
\end{align*}
for all $z = h(x) \in h(U)$. 
On the other hand, formulas \eqref{eq_linear-conj-1} and \eqref{eq_linear-conj-2} suggests that for all $i,j \in \{1,\dots, n\}$, and any source complex $y$,
\[
  c^y  \sum_{y' \colon y\to y'\in \cR} c_i^{-1} \k_{y\to y'}(y_i' - y_i) =   \sum_{y' \colon y\to y'\in \cR'} c^y \beta_{y\to y'} (y_i' - y_i) = \sum_{y' \colon y\to y'\in \cR'} \k'_{y\to y'} (y_i' - y_i)
\]
and
\[
  c^y  \sum_{y' \colon y\to y'\in \cR} c_i^{-1} c_j^{-1} \k_{y \to y'}(y_i' - y_j) \cdot(y'_j - y_j) =  \sum_{y' \colon y\to y'\in \cR'} \k'_{y\to y} (y_i' - y_i) \cdot (y_j' - y_j).   
\]
As a consequence, $L^h$ coincides with $L'$ the generator for $X'_{t \wedge \tau'} (h(x)) = X'_{t \wedge \tau'} (z)$ 
\begin{align*}
  L' f (z) 
  = & \sum_{i = 1}^n  \sum_{y \in \cC} z^{y} \sum_{y' \colon y \to y'  \in \cR'} \k_{y \to y'}' (y'_{i} - y_{i})  \frac{\partial}{\partial z_i} f(z) 
  +\frac{1}{2}\sum_{i,j = 1}^n   \sum_{y \in \cC} z^{y} \sum_{y' \colon y \to y' \in \cR'}  \k_{y \to y'}' (y_{i}' - y_{i}) \cdot(y_{j}' - y_{j})  \frac{\partial^2}{\partial z_i\partial z_j}f(z).
\end{align*}
This completes the proof of Proposition \ref{prop_linear-conj-rate}.
\qed

\section*{Code availability}
A Julia implementation for assessing reaction-identifiability, confoundability, and linear conjugacy with respect to both ODEs and SDEs is available at: 
 \url{https://github.com/faullouis/Identifiability-of-SDEs-for-RNs}.

 See the Julia documentation at \url{https://docs.julialang.org/}
 for a comprehensive user guide to the Julia programming language.

\section*{Acknowledgements}

We thank Christian Mazza, Enrico Bibbona, Peter Pfaffelhuber and Jan van Waaij for helpful discussions. LH was partially supported by the Swiss National Science Foundation grant (P2FRP2 188023).

\section*{Competing Interests}
The authors declare that they have no competing interests.

\providecommand{\bysame}{\leavevmode\hbox to3em{\hrulefill}\thinspace}
\providecommand{\MR}{\relax\ifhmode\unskip\space\fi MR }
\providecommand{\MRhref}[2]{%
  \href{http://www.ams.org/mathscinet-getitem?mr=#1}{#2}
}
\providecommand{\href}[2]{#2}


\end{document}